\documentclass[12pt, reqno]{amsart}
\usepackage[margin=1in]{geometry}
\usepackage{amssymb, amsmath, amsthm, amsfonts}
\usepackage{mathrsfs}
\usepackage[all]{xy}
\usepackage{multirow}
\usepackage{url}
\newcommand{\GG}[1]{}
\newcommand{\noop}[1]{}

\newtheorem{theorem}{Theorem}[section]
\newtheorem{lemma}[theorem]{Lemma}
\newtheorem{proposition}[theorem]{Proposition}
\newtheorem{corollary}[theorem]{Corollary}
\usepackage{tikz}
\usetikzlibrary{positioning}
\usepackage{verbatim}

\newtheorem{predefinition}[theorem]{Definition}

\newtheorem{preremark}[theorem]{Remark}

\newtheorem{prenotation}[theorem]{Notation}

\newtheorem{preexample}[theorem]{Example}

\newtheorem{preclaim}[theorem]{Claim}

\newtheorem{prequestion}[theorem]{Question}

 \makeatletter

\usepackage[hidelinks]{hyperref}

\DeclareMathOperator{\Gal}{Gal}

\DeclareMathOperator{\ord}{ord}

\DeclareMathOperator{\frakp}{\mathfrak{p}}

\DeclareMathOperator{\Gg}{G_{glob}}
\DeclareMathOperator{\Gl}{G_{loc}}

\DeclareMathOperator{\cl}{Cl}
\DeclareMathOperator{\kg}{\kappa_{glob}}
\DeclareMathOperator{\kl}{\kappa_{loc}}

\DeclareMathOperator{\Fg}{F_{glob}}
\DeclareMathOperator{\Fl}{F_{loc}}
\DeclareMathOperator{\Hg}{H_{glob}}
\DeclareMathOperator{\Hl}{H_{loc}}
\DeclareMathOperator{\dl}{\delta_{loc}}
\DeclareMathOperator{\dg}{\delta_{glob}}

\let\div\relax
\DeclareMathOperator{\div}{div}

\newcommand{\kn}{\mathfrak{K}}

\newcommand{\bbF}{\mathbb{F}}
\newcommand{\bbZ}{\mathbb{Z}}
\newcommand{\bbA}{\mathbb{A}}

\newcommand{\cO}{\mathcal{O}}

\newcommand{\frp}{\mathfrak{p}}
\newcommand{\frq}{\mathfrak{q}}
\newcommand{\frr}{\mathfrak{r}}
\newcommand{\frn}{\mathfrak{n}}

\newcommand{\fra}{\mathfrak{a}}

\title[The Hasse Norm Principle in Function Fields]{The Hasse Norm Principle in Global Function Fields}
\date{}

\author[M\^{a}nz\u{a}\c{t}eanu]{Adelina M\^{a}nz\u{a}\c{t}eanu}
\address{Adelina M\^{a}nz\u{a}\c{t}eanu\\ Mathematisch Instituut\\ Niels Bohrweg 1, 2333 CA Leiden, Netherlands}
\email{m.manzateanu@math.leidenuniv.nl}

\author[Newton]{Rachel Newton}
\address{Rachel Newton\\ Department of Mathematics and Statistics\\ University of Reading\\ White\-knights\\ PO Box 220\\ Reading RG6 6AX\\ UK}
\email{r.d.newton@reading.ac.uk}

\author[Ozman]{Ekin Ozman}
\address{Ekin Ozman, Bogazici University, Faculty of Arts and Sciences, Bebek, Istanbul, 34342, Turkey}
\email{ekin.ozman@boun.edu.tr}

\author[Sutherland]{Nicole Sutherland}
\address{Nicole Sutherland, Computational Algebra Group, School of Mathematics and Statistics, The University of Sydney, 2006, Australia}
\email{nicole.sutherland@sydney.edu.au}

\author[Uysal]{Rabia G\"ul\c{s}ah Uysal}
\address{Rabia G\"ul\c{s}ah Uysal, Department of Mathematics, Middle East Technical University, Ankara, 06800, Turkey}
\email{gulsah.uysal@metu.edu.tr}

\begin{document}

\begin{abstract} 
Let $L$ be a finite extension of $\bbF_q(t)$.
We calculate the proportion of polynomials of degree $d$ in $\bbF_q[t]$ that are everywhere locally norms from $L/\bbF_q(t)$ which fail to be global norms from $L/\bbF_q(t)$.
\end{abstract}




\maketitle

\section{Introduction} 
\label{Introduction}

The Hasse norm principle is said to hold for an extension of global fields $L/k$ if the knot group
 \[\kn(L/k) =\frac{k^{{\times}} \cap N_{L/k} \bbA_L^\times}{ N_{L/k} L^{{\times}}}\]
is trivial, in other words if an element of $k^\times$ is a global norm from $L/k$ if and only if it is a norm everywhere locally. Hasse's original norm theorem \cite{Hasse} shows that the Hasse norm principle holds for cyclic extensions of number fields. Since then, there have been several research articles giving methods for computing knot groups and sufficient criteria for the Hasse norm principle to hold, see \cite{Bae_Jung, Bartels1, Bartels2, Drakokhrust_Platonov, Gerth, Gurak1, Gurak2, Horie, Hoshi_Kanai_Yamasaki, Kagawa, MacedoAn, Macedo_Newton, Razar, DashengWei}, for example. 
Furthermore, new breakthroughs obtained when studying arithmetic objects in families mean there has been a great deal of interest in the frequency of failure of local-global principles -- see \cite{browning_survey} for a survey of recent progress. In particular, the frequency of failure of the Hasse norm principle for number fields has been studied in \cite{browning_newton, FLN, FLN2, MacedoD4, Rome}. 

In this paper, we study failures of the Hasse norm principle in the global function field setting. Let $q$ be a power of a prime $p$, let $L/\bbF_q(t)$ be a finite extension with full constant field $\bbF_{q^f}$ and let $\frn\subset\bbF_q[t]$ be an ideal. In order to compare the number of global norms from $L/\bbF_q(t)$ with the number of everywhere local norms, we define counting functions
\begin{eqnarray*}
N_{\textup{glob}}(L/\bbF_q(t),\frn, d)&=&\#\{\alpha\in \bbF_q[t] \cap N_{L/\bbF_q(t)}L^\times\mid (\alpha,\frn)=1, \deg \alpha=d\},\ \ \textrm{  and}\\
N_{\textup{loc}}(L/\bbF_q(t),\frn, d)&=&\#\{\alpha\in  \bbF_q[t]\cap N_{L/\bbF_q(t)}\bbA_L^\times\mid (\alpha,\frn)=1,  \deg \alpha=d\}.
\end{eqnarray*}

The following constant will play an important role in our results:
\begin{equation} \label{eq:h}
h=\gcd\{\deg \frp \mid \frp \textrm{ infinite place of } L \}.
\end{equation}

We may now state our main theorem:
\begin{theorem}\label{thm:main}
We have
\[\lim_{\substack{d\to \infty \\ fh\mid d}}\frac{N_{\textup{glob}}(L/\bbF_q(t),\frn, d)}{N_{\textup{loc}}(L/\bbF_q(t),\frn, d)}=\frac{1}{\# \mathfrak{K}(L/\bbF_q(t))}, \]
where the limit is taken over degrees $d$ such that $fh\mid d$.
\end{theorem}
In the special case $\frn=\bbF_q[t]$, Theorem~\ref{thm:main} is an integral analogue of \cite[Theorem~1.2]{browning_newton} in the function field setting. We note that examples where the knot group is non-trivial certainly exist in this setting: for example, \cite[\S11.4]{Tate} shows that the knot group is $\bbZ/2\bbZ$ for the biquadratic extension $\bbF_5(\sqrt{t},\sqrt{t+1})/\bbF_5(t)$ since all its decomposition groups are cyclic.

    In order to obtain Theorem~\ref{thm:main}, we show that the method of Cohen and Odoni can be used to prove the following local version of \cite[Theorem~IIB]{cohen_odoni}:

\begin{theorem}
\label{IIB_el}
There exists a finite abelian extension $L_{\textup{loc}}/L$ 
with the following properties:
\begin{enumerate}
\item[(a)] if $d$ is a large multiple of $fh$, then $N_{\textup{loc}}(L/\bbF_q(t),\frn, d)$ is asymptotically
\begin{equation}\label{IIB}
h \kl C \frac{q^d d^{B-1}}{[L_{\textup{loc}}:L]} \lambda_{\mathfrak{n}}^{-1} \{1+ O\bigl(d^{-A} \omega^4(\mathfrak{n})\bigr)\} + 
O\Bigl(q^{d/2} e^{2 \sqrt{d \omega(\mathfrak{n})}}\Bigr)
\end{equation}
where $A, B$ and $C$ are positive constants depending only on $L/\bbF_q(t)$ and $0<B<1$, $\omega(\frn)$ is the number of distinct prime divisors of $\frn$ and
   \[\lambda_\frn=\prod_{\frp\mid\frn}\{1+\delta(\frp)q^{-\deg\frp}+\delta(\frp^2)q^{-2\deg\frp}+\dots\}\] 
   where $\delta$ is the indicator function for norms of fractional ideals of $L$, see Section~\ref{sec:pf};
\item[(b)]
if $d$ is not a multiple of $fh$, then $N_{\textup{loc}}(L/\bbF_q(t),\frn, d)$ is only
$$O\bigl(q^d d^{B-1-A} \omega^4(\mathfrak{n})\bigr) + O\Bigl(q^{d/2} e^{2 \sqrt{d \omega(\mathfrak{n})}}\Bigr)$$
\end{enumerate}
where the constants involved in the $O$ symbols may be taken uniform in $d$ and 
$\mathfrak{n}$.
\end{theorem}

The constant $\kl$ and its global analogue $\kg$ are defined as follows:
\begin{equation}\label{eq:kappas}
\kl=\#(\mathbb{F}_q^{\times}\cap N_{L/\bbF_q(t)}\mathbb{A}_L^{\times})\ \ \textrm{and}\ 
\kg=\#(\mathbb{F}_q^{\times}\cap N_{L/\bbF_q(t)}L^{\times}).
\end{equation}

One key difference with the number field case handled in \cite{browning_newton} is the special role played by the constant fields in the function field setting. A key step in our proof of Theorem~\ref{thm:main} is to show that $L_{\textup{loc}}$ and its global analogue $L_{\textup{glob}}$ both have full constant field $\bbF_{q^{fh}}$. This is achieved in Theorem~\ref{thm:constants} using the following result, which is
proved in Section~\ref{sec:constant}:

\begin{theorem}\label{thm:hes}
Let $F$ be a global function field with full constant field $\bbF_q$, let $\mathfrak{m}$ be an effective divisor of $F$ and let $H$ be a finite index subgroup of the 
ray class group $\cl_{\mathfrak{m}}(F)$. Then the ray class field corresponding to $H$ has full constant field $\bbF_{q^r}$, where $r$ is the smallest positive degree of a divisor in $H$.
\end{theorem}

To obtain an analogue of Theorem~\ref{thm:main} for rational functions rather than polynomials, one would need to handle sums over fractional ideals written as quotients of coprime integral ideals, in a similar fashion to what was done at the bottom of p.343 of \cite{browning_newton}. The appearance of $\omega(\frn)$ in the error terms of \eqref{IIB} means that these error terms would need to be handled carefully, but we believe this should be possible with some work.

\subsection*{Acknowledgements}
This project began at the \emph{Women in Numbers Europe 3} workshop in Rennes, August 2019. We are grateful to the organisers for bringing us together and providing us with an excellent working environment to get this project underway.
 We thank Alp Bassa, Titus Hilberdink, Yiannis Petridis and Efthymios Sofos for useful discussions. 
{\sc Magma}~\cite{magma224} was used to investigate examples. Rachel Newton is supported by EPSRC grant EP/S004696/1.

\section{Reducing to the separable case} \label{Lemmas}
One major difference between the function field setting and the number field setting is the presence of inseparable 
extensions in the function field case. Fortunately, Cohen and Odoni \cite{cohen_odoni} give the following lemma allowing us to reduce to the case of a separable extension:

\begin{lemma}[{\cite[Lemma~1.1]{cohen_odoni}}]\label{lem:sep}
	Let $F$ be a perfect field of characteristic $p\neq0$. If $t$ is an indeterminate and L is a finite extension of $F(t)$ of degree of inseparability $p^{i}$, then $L=KM$, where $K=F(t^{p^{-i}})$ and $M$ is the maximal subfield of $K$ separable over $F(t)$; in particular, $L/K$ is separable.
	\end{lemma}

  Lemma~\ref{lem:sep2} below allows us to transport the property of being a (global or everywhere local) norm from $L/\bbF_q(t)$ to the separable extension $L/K$ given by Lemma~\ref{lem:sep} and back.  Before stating it, we explain what we mean by a fractional ideal of $L$ and describe the correspondence between fractional ideals and finite divisors.

Let $L/\bbF_q(t)$ be a finite extension. Write $\cO_L$ for the integral closure of $\bbF_q[t]$ in $L$. Note that, unlike in the number field case, $\cO_L$ is not canonical -- it depends on a choice of generator $t$ for $\bbF_q(t)/\bbF_q$. We consider the choice of generator $t$ to be fixed throughout this paper. By a fractional ideal of $L$, we mean a fractional ideal of $\cO_L$. For $\alpha\in L^\times$, we write $(\alpha)$ for the principal fractional ideal of $\cO_L$ generated by $\alpha$.

The infinite place of $\bbF_q(t)$ corresponds to the valuation $\ord_\infty$ on $\bbF_q(t)$ given by $\ord_\infty\bigl(\frac{f(t)}{g(t)}\bigr)=\deg g(t)-\deg f(t)$ for $f(t), g(t)\in\bbF_q[t]$. In other words, the infinite place of $\bbF_q(t)$ corresponds to the prime ideal generated by $\frac{1}{t}$ in $\bbF_q[\frac{1}{t}]$. We write $\infty$ for the infinite place of $\bbF_q(t)$. 

Let $D(L)$ denote the group of divisors of $L$ and let $D_\infty(L)$ denote the subgroup of finite divisors, meaning those whose support does not include any place above $\infty$. We identify the finite places of $L$ with the nonzero prime ideals of $\cO_L$ (see \cite[\S5.2]{koch}, for example). Thus, since $\cO_L$ is a Dedekind domain, the map
\[\sum_i a_i \frp_i \to \prod_i \frp_i^{a_{i}} \]
allows us to identify $D_{\infty}(L)$ with the multiplicative group of nonzero fractional ideals of $\cO_L$, which we will denote by $I_L$. Having made this identification, we will refer to the degree of a fractional ideal, meaning the degree of the associated divisor.

\begin{lemma} \label{lem:sep2}
Let $L/\bbF_q(t)$ be a finite extension of degree of inseparability $p^{i}$, let $K=\bbF_q(t^{p^{-i}})$ and
let $\alpha\in \bbF_q(t)$. Then
	\begin{enumerate}
		\item \label{ideal} the fractional ideal $(\alpha)$ of $\bbF_q[t]$ is the $L/\bbF_q(t)$ norm of some fractional ideal of $\cO_L$ if and only if the fractional ideal $(\alpha^{p^{-i}})$ of $\cO_K$ is the $L/K$ norm of some fractional ideal of $\cO_L$;
		\item \label{glob} $\alpha\in N_{L/\bbF_q(t)} L^{{\times}}$ if and only if $\alpha^{p^{-i}} \in N_{L/K} L^{{\times}}$;
		\item \label{loc}  $\alpha\in N_{L/\bbF_q(t)} \bbA_{L}^{{\times}} $ if and only if $\alpha^{p^{-i}}\in N_{L/K} \bbA_{L}^{{\times}} $.
			\end{enumerate}
\end{lemma}

	\begin{proof} Parts \eqref{ideal} and \eqref{glob} are the content of \cite[Lemma 1.2]{cohen_odoni}. We prove \eqref{loc}. First suppose that $\alpha^{p^{-i}}\in N_{L/K}\bbA_{L}^{{\times}}$. This means that for every place $\frq$ of $K$ there exists $(\beta_{\frr})_\frr \in \prod_{\frr\mid \frq}L_{\frr}^{{\times}}$ such that 
	\begin{equation}\label{eq:locnorm}
	\alpha^{p^{-i}}= \prod\limits_{\frr \mid \frq}N_{L_{\frr}/K_{\frq}} (\beta_{\frr} ).
	\end{equation}
	Let $\frp$ be a place of $\bbF_q(t)$.
	 By \cite[Lemma~7.3]{rosen}, since $K/\bbF_q(t)$ is a purely inseparable extension, there is a unique place $\frq$ of $K$ above $\frp$. Taking $N_{K/\bbF_q(t)}$ of both sides of \eqref{eq:locnorm} gives
	  \begin{equation}\label{eq:locnorm2}
	N_{K/\bbF_q(t)}(\alpha^{p^{-i}})= \prod\limits_{\frr\mid \frq}N_{K_\frq/(\bbF_q(t))_\frp}( N_{L_{\frr}/K_{\frq}} (\beta_{\frr} ))= \prod\limits_{\frr\mid \frp}N_{L_{\frr}/(\bbF_q(t))_\frp}(\beta_{\frr} ).
	\end{equation}
	Now observe that $N_{K/\bbF_q(t)}(\alpha^{p^{-i}})=\alpha$, since $K/\bbF_q(t)$ is a purely inseparable extension of degree $p^{i}$. Hence \eqref{eq:locnorm2} becomes 
	  \begin{equation}\label{eq:locnorm3}
	  \alpha =\prod\limits_{\frr \mid \frp}N_{L_{\frr}/(\bbF_q(t))_\frp}(\beta_{\frr} ).
	    \end{equation}
Since $\frp$ was arbitrary, we have shown that $\alpha\in N_{L/\bbF_q(t)} \bbA_{L}^{{\times}} $, as required.	
	
	 Now suppose that $\alpha\in N_{L/\bbF_q(t)} \bbA_{L}^{{\times}} $, so for every place $\frp$ of $\bbF_q(t)$ there exists $(\beta_{\frr})_\frr \in \prod_{\frr \mid \frp}L_{\frr}^{{\times}}$ such that 
	\begin{equation}\label{eq:locnorm4}
	\alpha= \prod\limits_{\frr \mid \frp}N_{L_{\frr}/(\bbF_q(t))_\frp} (\beta_{\frr} ).
	\end{equation}
	 Again, for each place $\frp$ of $\bbF_q(t)$ there exists a unique place $\frq$ of $K$ above $\frp$. Furthermore, $N_{K_\frq/(\bbF_q(t))_\frp}(x)=x^{p^i}$ for all $x\in K_\frq$.
	  Thus, \eqref{eq:locnorm4} becomes
	 \begin{equation}\label{eq:locnorm5}
	\alpha= \prod\limits_{\frr \mid \frq}( N_{L_{\frr}/K_{\frq}} (\beta_{\frr} ))^{p^i}.
	\end{equation}
	 Hence $\alpha^{p^{-i}}\in N_{L/K} \bbA_{L}^{{\times}} $, as required.
\end{proof}

Lemma~\ref{lem:sep2} shows that $\alpha\mapsto\alpha^{p^{-i}}$ gives bijections
\[\{\alpha\in  \bbF_q[t]\cap N_{L/\bbF_q(t)}L^\times\mid (\alpha,\frn)=1,  \deg \alpha=d\}\to \{\beta\in \cO_K \cap N_{L/K}L^\times\mid (\beta,\frn)=1,  \deg \beta=d\}\]
and
\[\{\alpha\in  \bbF_q[t]\cap N_{L/\bbF_q(t)}\bbA_L^\times\mid (\alpha,\frn)=1,  \deg \alpha=d\}\to \{\beta\in \cO_K \cap N_{L/K}\bbA_L^\times\mid (\beta,\frn)=1,  \deg \beta=d\}\]
where $\cO_K=\bbF_q[t^{p^{-i}}]$ and $\deg\beta$ is the degree with respect to the variable $t^{p^{-i}}$. Defining
\begin{eqnarray*}
N_{\textup{glob}}(L/K,\mathfrak{n}, d)&=&\#\{\beta\in \cO_K \cap N_{L/K}L^\times\mid (\beta,\mathfrak{n})=1,  \deg \beta=d\},\ \ \textrm{  and}\\
N_{\textup{loc}}(L/K,\mathfrak{n}, d)&=&\#\{\beta\in \cO_K \cap N_{L/K}\bbA_L^\times\mid (\beta,\mathfrak{n})=1,  \deg \beta=d\}
\end{eqnarray*}
 gives
\begin{eqnarray}
N_{\textup{glob}}(L/\bbF_q(t),\frn, d)&=& N_{\textup{glob}}(L/K,\frn, d),\ \ \textrm{  and} \label{eq:N(L/K)glob}\\
N_{\textup{loc}}(L/\bbF_q(t),\frn, d)&=&N_{\textup{loc}}(L/K,\frn, d).\label{eq:N(L/K)loc}
\end{eqnarray}
This allows us to restrict to the finite separable extension $L/K$ in order to prove Theorems~\ref{thm:main} and \ref{IIB_el}. We now list two further consequences of Lemma~\ref{lem:sep2} that will be used in the proofs of our main results.

\begin{corollary}\label{cor:kappas}
In the setting of Lemma~\ref{lem:sep2}, we have
\[\mathbb{F}_q^{\times}\cap N_{L/\bbF_q(t)}\mathbb{A}_L^{\times}=\mathbb{F}_q^{\times}\cap N_{L/K}\mathbb{A}_L^{\times}\]
and
\[\mathbb{F}_q^{\times}\cap N_{L/\bbF_q(t)}L^{\times}=\mathbb{F}_q^{\times}\cap N_{L/K}L^{\times}.\]
\end{corollary}
		
\begin{proof}
This follows from Lemma~\ref{lem:sep2}, since $\alpha\mapsto \alpha^{p^{-i}}$ is an automorphism of $\bbF_q^\times$.
\end{proof}		

\begin{corollary}\label{cor:knot}
	In the setting of Lemma~\ref{lem:sep2}, the map $\alpha\mapsto\alpha^{p^{-i}}$ induces an isomorphism \[\kn(L/\bbF_q(t)) \xrightarrow{\sim} \kn(L/K).\]
\end{corollary}
\begin{proof}
This follows immediately from Lemma~\ref{lem:sep2}.	
\end{proof}

\section{Proof of our main results}\label{sec:pf}

In order to prove Theorem~\ref{IIB_el}, we will adapt the strategy of Cohen and Odoni in \cite{cohen_odoni} to the case of everywhere local norms. 
Define indicator functions on $I_K$ as follows: 
\begin{eqnarray*}
\delta(\fra)&=&\begin{cases} 1& \textrm{if } \fra\in N_{L/K} I_L,\\
0 & \textrm{otherwise,}
\end{cases}
\\
\dl(\fra)&=&\begin{cases} 1& \textrm{if } \fra=(\beta) \textrm{ for some }\beta\in K^\times\cap N_{L/K}\bbA_L^\times,\\
0 & \textrm{otherwise},\end{cases}\\
\dg(\fra)&=&\begin{cases} 1& \textrm{if } \fra=(N_{L/k}(\alpha)) \textrm{ for some }\alpha\in L^\times,\\
0 & \textrm{otherwise,}\end{cases}.
\end{eqnarray*}

\begin{lemma}We have
\[N_{\textup{loc}}(L/\bbF_q(t),\frn, d)=\kl\sum_{\substack{\fra\subset\cO_K\\ (\fra,\frn)=1\\ \deg\fra=d}}\dl(\fra)\]
and
\[N_{\textup{glob}}(L/\bbF_q(t),\frn, d)=\kg\sum_{\substack{\fra\subset\cO_K\\ (\fra,\frn)=1\\ \deg\fra=d}}\dg(\fra).\]
\end{lemma}

\begin{proof}
The terms $\kl$ and $\kg$ are there to account for the difference between elements of $\cO_K$ and principal integral ideals of $\cO_K$. Now the result follows from \eqref{eq:N(L/K)glob} and \eqref{eq:N(L/K)loc}.
\end{proof}

The next step is to show that the ideal generated by an everywhere local norm from $L/K$ is the norm of a fractional ideal of $\cO_L$. This is the content of Corollary~\ref{cor:eln} below.

\begin{lemma}\label{lem:norm}
Let $\alpha \in K$. Then $(\alpha) \in N_{L/K}(I_{L})$ if and only if for every finite place $\mathfrak{p}$ the greatest common divisor of the residue degrees $f_{\mathfrak{q}/\mathfrak{p}}$ of the places $\mathfrak{q}$ above  $\mathfrak{p}$ divides $\ord_{\mathfrak{p}}(\alpha)$. 
\end{lemma}
\begin{corollary}\label{cor:eln}
	If $\alpha \in K^{{\times}} \cap N_{L/K} (\bbA_{L}^{{\times}} ) $ then  $(\alpha) \in N_{L/K}(I_{L})$.

\end{corollary}

\begin{proof}[Proof of Lemma \ref{lem:norm} and Corollary \ref{cor:eln}.] Lemma \ref{lem:norm} and Corollary \ref{cor:eln} are the global function field analogues of \cite[Lemma 2.1]{browning_newton} and \cite[Corollary 2.2]{browning_newton}. The same proofs work.
\end{proof}

Using Lemma~\ref{lem:sep2} to move between $L/\bbF_q(t)$ and $L/K$, Corollary~\ref{cor:eln} means that a first approximation for $N_{\textup{loc}}(L/\bbF_q(t),\frn, d)$ is given by
\begin{equation}\label{eq:idealnorm}
\sum_{\substack{\fra\subset\cO_K\\ (\fra,\frn)=1\\ \deg\fra=d}}\delta(\fra)
\end{equation}
which counts integral ideals of $\bbF_q[t]$, coprime to $\frn$ and of degree $d$, that are norms of fractional ideals of $\cO_L$. In \cite[Theorem~IIA]{cohen_odoni}, Cohen and Odoni give an asymptotic formula for \eqref{eq:idealnorm} by studying the Dirichlet series
\[f(\mathfrak{n},t)=\sum_{\substack{\mathfrak{a}\subset \mathcal{O}_K\\ (\mathfrak{a},\mathfrak{n})=1}}\delta(\mathfrak{a})t^{\deg(\mathfrak{a})},\ \ \ \ \ \ |t|<q^{-1}.\]
 They then go on to analyse the behaviour of the Dirichlet series
\[
f_{\textup{glob}}(\frn,t)=\sum_{\substack{\mathfrak{a}\subset \mathcal{O}_K\\ (\mathfrak{a},\frn)=1}}\dg(\mathfrak{a})t^{\deg(\mathfrak{a})},\ \ \ \ \ \ |t|<q^{-1},\]
by expressing $\dg$ in terms of $\delta$ and a sum over the characters of a certain finite abelian group coming from class field theory. With some work, this allows them to deduce an asymptotic formula for $N_{\textup{glob}}(L/\bbF_q(t),\frn, d)$ in \cite[Theorem~IIB]{cohen_odoni}. We seek to employ the same strategy to analyse the behaviour of the Dirichlet series 
\[f_{\textup{loc}}(\mathfrak{n},t)=\sum_{\substack{\mathfrak{a}\subset \mathcal{O}_K\\ (\mathfrak{a},\mathfrak{n})=1}}\dl(\mathfrak{a})t^{\deg(\mathfrak{a})},\ \ \ \ \ \ |t|<q^{-1},\]
and thereby prove Theorem~\ref{IIB_el}. This requires us to express $\dl$ in terms of $\delta$ and a sum over the characters of a finite abelian group. This is achieved in Lemma~\ref{lem:indicators} after some class field theoretic preliminaries.

\subsection{Class field theory}  \label{CFT}
We begin by recalling some essential facts.
Let $\mathfrak{m}$ be an effective divisor of a global function field $F$. Let $D_{\mathfrak m}(F)$ denote the group of divisors of $F$ with support disjoint from the support of $\mathfrak m$. Write $P_{\mathfrak m}(F)$ for the subgroup of $D_{\mathfrak m}(F)$ consisting of
principal divisors $\div(f)$ such that $f\in F^\times$ satisfies $\ord_\frp(f-1)\geq \ord_{\frakp}  \mathfrak{m}$ for all places $\frakp$ in the support of $\mathfrak{m}$. The ray class group of $F$ modulo $\mathfrak m$ is defined to be
\[\cl_{\mathfrak m}(F) = D_{\mathfrak m}(F)/P_{\mathfrak m}(F).\]
The group $\cl_{\mathfrak m}(F) $ is never finite. However, its degree zero part \[\cl_{\mathfrak m}^{0}(F)= \{[\mathfrak{d}] \in \cl_{\mathfrak m} (F) \mid \deg \mathfrak{d}=0\}\] is finite, see \cite[p.139]{rosen}, for example.

  Class field theory gives a one-to-one correspondence between  the subgroups of finite index of the ray class group $\cl_{\mathfrak{m}}(F)$ and the finite abelian extensions of $F$ that are unramified away from $\mathfrak m$. The correspondence is via the Artin map which gives a canonical isomorphism $A_{E/F}:\cl_{\mathfrak{m}}(F)/H\xrightarrow{\sim}  \Gal(E/F)$, where $E/F$ is the extension associated to the subgroup $H$.
In particular, the places that split completely in $E/F$ are precisely the
places in $H$. 

We expect that the following proposition is well known, but we give the proof here for completeness.
\begin{proposition}\label{prop:finite}
Let $F$ be a global function field, let $\mathfrak{m}$ be an effective divisor of $F$ and let $H$ be a subgroup of the ray class group $\cl_{\mathfrak{m}}(F)$. Then $H$ has finite index in $\cl_{\mathfrak{m}}(F)$ if and only if $H$ contains a divisor class of nonzero degree.
\end{proposition}

\begin{proof}
Let $n$ be the smallest non-negative  degree of
a divisor class in $H$ and consider the following commutative diagram with exact rows:
\[
\xymatrix{0\ar[r]& \cl_{\mathfrak m}^0(F)\cap H\ar[r]\ar[d]& H\ar[r]^{\deg}\ar[d] & n\mathbb{Z}\ar[r]\ar[d] &0\\
0\ar[r]& \cl_{\mathfrak m}^0(F)\ar[r]& \cl_{\mathfrak m}(F)\ar[r]^{\deg} &\mathbb{Z}\ar[r] & 0.\\
}\]
The degree map in the bottom row is surjective since $\cl_{\mathfrak m}(F)$ surjects onto $\cl_{\mathfrak n}(F)$ for any $\mathfrak n \mid \mathfrak m$. In particular, $\cl_{\mathfrak m}(F)$ surjects onto the class group of $F$ \cite[Thm 1.7]{Milne} and it is well known that the degree map from the class group surjects onto $\mathbb Z$.
Now the snake lemma gives an exact sequence
\[0\to \frac{\cl_{\mathfrak m}^0(F)}{\cl_{\mathfrak m}^0(F)\cap H}\to \frac{\cl_{\mathfrak m}(F)}{H}\to \mathbb{Z}/n\mathbb{Z}\to 0.\]
Since $\cl_{\mathfrak m}^{0}(F)$ is finite, we deduce that $\cl_{\mathfrak m}(F)/H$ is finite if and only if $n\neq 0$.
\end{proof}

Now define two subgroups of $I_L$: 

\[\Hg=\{\mathfrak a \in I_L\mid N_{L/K} \mathfrak a =(N_{L/K}(\alpha)) \textrm{ for some }\alpha\in L^{\times}\}\]
and 
\[\Hl=\{\mathfrak a \in I_L\mid N_{L/K} \mathfrak a=(\beta) \textrm{ for some }\beta\in K^{\times}\cap N_{L/K}\mathbb{A}_L^{\times}\}.\]

In~\cite[\S3]{cohen_odoni}, Cohen and Odoni show that 
\[P_{\infty}(L)=\{(\beta)\in I_L \mid \beta\equiv 1\pmod{\frp} \ \ \forall \frp  \mid\infty\}\subset \Hg.\] 
They also show that $\Hg$ contains an ideal
of nonzero degree (see Lemma~\ref{lem:Hg contains deg h} for a proof that $\Hg$ contains an ideal
of degree $h$). Proposition~\ref{prop:finite} therefore shows that $\Hg$ defines a ray class field $L_{\textup{glob}}/L$ unramified outside the infinite places with $\Gal(L_{\textup{glob}}/L)=I_L/\Hg$.
Since $N_{L/K}L^\times\subset K^{\times}\cap N_{L/K}\mathbb{A}_L^{\times}$ we have $\Hg\subset \Hl$. Therefore, 
$\Hl$ defines a ray class field $L_{\textup{loc}}\subset L_{\textup{glob}}$ unramified outside the infinite places with $\Gal(L_{\textup{loc}}/L)=I_L/\Hl$.

\begin{lemma}\label{lem:altgps}
The norm map $N_{L/K}$ gives isomorphisms
\[I_L/\Hg \xrightarrow{\sim} \frac{N_{L/K}I_L}{\{(N_{L/K}(\alpha)) \mid \alpha\in L^{\times}\}}\] 
and
\[I_L/\Hl \xrightarrow{\sim} \frac{N_{L/K}I_L}{\{(\beta)\mid \beta\in K^{\times}\cap N_{L/K}\mathbb{A}_L^{\times}\}}.\]
We denote the quotient groups on the right-hand sides by $\Gg$ and $\Gl$, respectively.
\end{lemma}

\begin{proof}
By Corollary~\ref{cor:eln}, $\{(\beta)\mid \beta\in K^{\times}\cap N_{L/K}\mathbb{A}_L^{\times}\}\subset N_{L/K}I_L$ so the second map is well defined. The rest is clear.
\end{proof}

The next lemma is a direct consequence of orthogonality of characters, as in \cite[\S3]{cohen_odoni}.

\begin{lemma}\label{lem:indicators}
For all $\mathfrak{a}\in I_K$,

\begin{eqnarray*}\dg(\mathfrak{a})&=&\frac{\delta(\mathfrak{a})}{\#\Gg}\sum_{\chi\in (\Gg)^\vee}\chi(\mathfrak{a}),\ \ \textrm{and}\\
\dl(\mathfrak{a})&=&\frac{\delta(\mathfrak{a})}{\#\Gl}\sum_{\chi\in (\Gl)^\vee}\chi(\mathfrak{a})
\end{eqnarray*}
where $G^\vee$ denotes the group of characters of an abelian group $G$.
\end{lemma}

Lemma~\ref{lem:indicators} has the following immediate consequence:

\begin{corollary}\label{cor:fs}
 For $|t|<q^{-1}$, 
\begin{eqnarray*}
f_{\textup{glob}}(\mathfrak{n},t)&=&\frac{1}{\#\Gg}\sum_{\chi\in (\Gg)^\vee}f(\mathfrak{n}, t,\chi),\ \ \textrm{and}\\
f_{\textup{loc}}(\mathfrak{n},t)&=&\frac{1}{\#\Gl}\sum_{\chi\in (\Gl)^\vee}f( \mathfrak{n}, t,\chi),\\
\textrm{where}\  f(\mathfrak{n}, t, \chi)&=&\sum_{\substack{\mathfrak{a}\subset \mathcal{O}_K\\ (\mathfrak{a},\mathfrak{n})=1}}\delta(\mathfrak{a})\chi(\mathfrak{a})t^{\deg(\mathfrak{a})}.
\end{eqnarray*}
\end{corollary}

Let $\Fg$ and $\Fl$ denote the degrees of the constant field extensions in $L_{\textup{glob}}/L$ and $L_{\textup{loc}}/L$, respectively.
Now \cite[Theorem~IIB]{cohen_odoni} shows that if $d$ is a large multiple of $f\Fg$, then $N_{\textup{glob}}(L/\bbF_q(t),\frn, d)$ is asymptotically
\begin{equation}\label{eq:glob}
\Fg \kg C \frac{q^d d^{B-1}}{[L_{\textup{glob}}:L]} \lambda_{\mathfrak{n}}^{-1} \{1+ O\bigl(d^{-A'} \omega^4(\mathfrak{n})\bigr)\} + 
O\Bigl(q^{d/2} e^{2 \sqrt{d \omega(\mathfrak{n})}}\Bigr)
\end{equation}
where $B$ and $C$ are as in Theorem~\ref{IIB_el} and $A'$ is a positive constant depending only on $L/\bbF_q(t)$. This result is proved using the expression for $f_{\textup{glob}}(\mathfrak{n},t)$ given in Corollary~\ref{cor:fs}. (To be completely accurate, we note that Cohen and Odoni give a superficially different expression for $f_{\textup{glob}}(\mathfrak{n},t)$ in \cite[(3.1)]{cohen_odoni}, owing to their use of $I_L/\Hg$ in place of the isomorphic group $\Gg$.) Employing the exact analogue of the proof of \cite[Theorem~IIB]{cohen_odoni} with $f_{\textup{loc}}(\mathfrak{n},t)$ in place of $f_{\textup{glob}}(\mathfrak{n},t)$ shows that
if $d$ is a large multiple of $f\Fl$, then $N_{\textup{loc}}(L/\bbF_q(t),\frn, d)$ is asymptotically
\begin{equation}\label{eq:IIB_el}
\Fl \kl C \frac{q^d d^{B-1}}{[L_{\textup{loc}}:L]} \lambda_{\mathfrak{n}}^{-1} \{1+ O\bigl(d^{-A} \omega^4(\mathfrak{n})\bigr)\} + 
O\Bigl(q^{d/2} e^{2 \sqrt{d \omega(\mathfrak{n})}}\Bigr)
\end{equation}
where $A,B$ and $C$ are as in Theorem~\ref{IIB_el}. Therefore, to complete the proof of Theorem~\ref{IIB_el}, it remains to show that $\Fl=h$, where $h$ is as defined in \eqref{eq:h}. In fact, we go further and prove in Theorem~\ref{thm:constants} that $\Fl=\Fg=h$.

\subsection{Constant fields}\label{sec:constant}

Recall from \eqref{eq:h} that
\[h=\gcd\{\deg \frp \mid \frp \textrm{ infinite place of } L \}.\]
Our main aim in this subsection is to complete the proof of Theorem~\ref{IIB_el} by proving the following result:
\begin{theorem}\label{thm:constants}
The full constant fields of $L_{\textup{glob}}$ and $L_{\textup{loc}}$ are both equal to $\bbF_{q^{fh}}$.
\end{theorem}
The first step towards the proof of Theorem~\ref{thm:constants} is to prove Theorem~\ref{thm:hes}. This requires the following result of Hess and Massierer:

\begin{lemma}[{\cite[Lemma~3.2]{hess-massierer}}]\label{lem:HM}
Let $F$ be a global function field with full constant field $\bbF_q$ and let $F'/F$ be a constant field extension of finite degree. Then $\Gal(F'/F)$ is generated by the Frobenius automorphism $\varphi$ and the Artin map
\[A_{F'/F} : D(F) \to \Gal(F'/F)\]
    is given by
    \[A_{F'/F}(\mathfrak{d})=\varphi^{\deg \mathfrak{d}}.\]
The zero divisor of $F$ is a modulus of $F'/F$.
\end{lemma}

\begin{proof}[Proof of Theorem~\ref{thm:hes}]
Let $E$ denote the ray class field corresponding to $H$ and suppose that the full constant field of $E$ is $\bbF_{q^s}$. Let $\mathfrak{d}$ be a divisor in $H$. Then $\mathfrak{d}$ is in the kernel of the Artin map for $E/F$. Therefore, $\mathfrak{d}$ is in the kernel of the Artin map for the constant subextension $\bbF_{q^s}F/F$ of degree $s$. By Lemma~\ref{lem:HM}, this implies that $s\mid\deg\mathfrak{d}$. We deduce that $s\mid r$, by the definition of $r$. We will complete the proof by showing that $r\mid s$. It suffices to show that $\bbF_{q^r}\subset E$. Let $\mathfrak{p}$ be a place in $H$, in other words a place that splits completely in $E/F$. Then $r\mid \deg\mathfrak{p}$, since $r$ is the greatest common divisor of the degrees of the divisors in $H$. Now Lemma~\ref{lem:HM} shows that $\frp$ splits completely in the degree $r$ constant extension $\bbF_{q^r}F/F$. Therefore, $\bbF_{q^r}\subset E$ by the Chebotarev density theorem.
\end{proof}

To complete the proof of Theorem~\ref{thm:constants} we need the following auxiliary results:

\begin{lemma}\label{lem:deg}
Let $L/\bbF_q(t)$ be a finite extension and let $\alpha\in L^\times$. Then
\[\deg(\alpha)=-\sum_{\frp \mid\infty}\ord_\frp \alpha\cdot \deg \frp.\]
\end{lemma}

\begin{proof}
Recall that by the degree of a fractional ideal of $\cO_L$, we mean the degree of the associated divisor of $L$, as explained in Section~\ref{Lemmas}. The divisor corresponding to $(\alpha)=\prod_{\frp\nmid\infty}\frp^{\ord_\frp\alpha}$ is $\sum_{\frp\nmid\infty}\ord_\frp\alpha\cdot \frp$. Moreover,
\[\div\alpha=\sum_{\frp}\ord_\frp\alpha\cdot \frp =\sum_{\frp\nmid\infty}\ord_\frp\alpha\cdot \frp+\sum_{\frp\mid\infty}\ord_\frp\alpha\cdot \frp .\]
Taking degrees yields the result since $\deg(\div\alpha)=0$.
\end{proof}

\begin{lemma}\label{lem:Hg contains deg h}
$\Hg$ contains an ideal of degree $h$.
\end{lemma}

\begin{proof}
Let $\frp_1, \dots, \frp_n$ be the infinite places of $L$ and let $a_1,\dots , a_n\in \bbZ$ be such that 
\begin{equation}\label{eq:degh}
\sum_{i=1}^n{a_i\deg \frp_i}=h.
\end{equation}
Choose $\alpha\in L^\times$ such that $\ord_{\frp_i}\alpha= -a_i$ for $i=1,\dots, n$. The principal fractional ideal $(\alpha)$ of $\cO_L$ is in $\Hg$ by definition of $\Hg$. It follows from Lemma~\ref{lem:deg} that $\deg(\alpha)=h$.
\end{proof}

\begin{lemma}\label{lem:h divides deg}
 Let $\mathfrak a\in \Hl$. Then $h \mid \deg \mathfrak a$.
\end{lemma}

\begin{proof}
Since $\mathfrak a \in \Hl$, there exists $\beta \in K^{\times} \cap N_{L/K} \mathbb{A}_L^{\times}$ with $N_{L/K} \mathfrak a = (\beta)$.
Write $\mathfrak a = \prod \mathfrak{q}_i^{a_i}$, where the $\mathfrak{q}_i$ are prime ideals in $\cO_L$ and the $a_i$ are integers. Now 
\begin{equation}\label{eq:beta}
(\beta)=N_{L/K} \mathfrak a = \prod N_{L/K} (\mathfrak{q}_i)^{a_i} =  \prod \frp_i^{a_i f_{\mathfrak{q}_i/\frp_i}}
\end{equation}
 where $\frp_i=\mathfrak{q}_i\cap \cO_K$. Recall that the full constant field of $K$ is $\bbF_q$ and the full constant field of $L$ is $\bbF_{q^f}$ so $f_{\mathfrak{q}_i/\frp_i}\deg\frp_i=f\deg\mathfrak{q}_i$. Now taking degrees in \eqref{eq:beta} gives
\begin{equation}\label{eq:deg}
\deg(\beta)=\sum a_i f_{\mathfrak{q}_i/\frp_i}\deg\frp_i=f\sum a_i\deg\mathfrak{q}_i=f\deg\fra.
\end{equation}
Since $\beta \in K^{\times} \cap N_{L/K} \mathbb{A}_L^{\times}$, for every place $\frp$ of $K$ there exists $(\gamma_{\frq})_\frq \in \prod_{\frq \mid \frp}L_{\frq}^{{\times}}$ such that 
	\begin{equation}\label{eq:betalocnorm}
	\beta= \prod\limits_{\frq \mid \frp}N_{L_{\frq}/K_{\frp}} (\gamma_{\frq} ).
	\end{equation}
Therefore, 
\begin{equation}
\ord_\frp \beta=\sum_{\frq \mid \frp}\ord_\frp(N_{L_{\frq}/K_{\frp}} (\gamma_{\frq} ))=\sum_{\frq \mid \frp}f_{\frq/\frp}\ord_\frq\gamma_\frq
\end{equation}
whereby Lemma~\ref{lem:deg} gives
\begin{equation}\label{eq:deg2}
\deg(\beta)= -\sum_{\frp \mid \infty}\ord_\frp \beta \cdot\deg \frp=-\sum_{\frp \mid \infty}\deg \frp\sum_{\frq\mid \frp}f_{\frq/\frp}\ord_\frq \gamma_\frq =-f\sum_{\frq \mid \infty}\ord_\frq \gamma_\frq \cdot \deg \frq.
\end{equation}
Combining \eqref{eq:deg} and \eqref{eq:deg2} gives 
\[\deg\fra=-\sum_{\frq \mid \infty}\ord_\frq \gamma_\frq \deg \frq.\]
By definition of $h$, we have $h\mid \deg \frq$ for all infinite places $\frq$ of $L$. Therefore, $h\mid\deg\fra$.
\end{proof}

\begin{corollary}\label{cor:h is gcd} We have
$h=\gcd\{\deg \mathfrak a \mid \mathfrak a \in \Hg\}=\gcd\{\deg \mathfrak a \mid \mathfrak a\in \Hl\}$.
\end{corollary}

\begin{proof}
Let $d_g= \gcd\{\deg \mathfrak a \mid \mathfrak a\in \Hg\}$ and $d_\ell=\gcd\{\deg \mathfrak a \mid \mathfrak a \in \Hl\}$.  By Lemma \ref{lem:Hg contains deg h}, $\Hg$ contains a ideal of degree $h$, whereby $d_g \mid h.$ Since $\Hg \subset \Hl,$ we also have $d_\ell \mid d_g$ and hence $d_\ell \mid h.$ By Lemma \ref{lem:h divides deg}, $h \mid \deg \fra$ for every $\fra \in \Hl$, whereby $h \mid d_{\ell}$ and hence $h=d_\ell=d_g.$ 
\end{proof}

Now Theorem~\ref{thm:constants} follows from Theorem~\ref{thm:hes} and Corollary~\ref{cor:h is gcd}. In addition, Theorem~\ref{IIB_el} follows from \eqref{eq:IIB_el} and Theorem~\ref{thm:constants}.

\subsection{Proof of Theorem~\ref{thm:main}}

By Theorem~\ref{thm:constants}, $L_{\textup{loc}}$ and $L_{\textup{glob}}$ both have full constant field $\bbF_{q^{fh}}$. Now taking the quotient of \eqref{eq:glob} by \eqref{eq:IIB_el} and letting $d\to \infty$ via multiples of $fh$ gives
\begin{equation}
\lim_{\substack{d\to \infty \\fh\mid d}}\frac{N_{\textup{glob}}(L/\bbF_q(t),\frn, d)}{N_{\textup{loc}}(L/\bbF_q(t),\frn, d)}=\frac{\kg}{\kl}\cdot \frac{ 1}{[L_{\textup{glob}}:L_{\textup{loc}}]}. 
\end{equation}
The following lemma completes the proof of Theorem~\ref{thm:main}:
\begin{lemma}\label{lem:knot}
The sequence
\[1\to \frac{\mathbb{F}_q^{\times}\cap N_{L/K}\mathbb{A}_L^{\times}}{\mathbb{F}_q^{\times}\cap N_{L/K}L^{\times}}\to\mathfrak{K}(L/K)\to \frac{\{(\beta)\mid \beta\in K^{\times}\cap N_{L/K}\mathbb{A}_L^{\times}\}}{\{(N_{L/K}(\alpha)) \mid \alpha\in L^{\times}\}}\to 1\]
is exact. Consequently, 
\[\# \mathfrak{K}(L/\bbF_q(t))  =\frac{\kl}{\kg}\cdot [L_{\textup{glob}}:L_{\textup{loc}}].\]
\end{lemma}

\begin{proof}
The right-hand map is given by $\beta\mapsto (\beta)$. The exactness of the sequence is easily verified. The right-hand term is the kernel of the natural surjection $\Gg\twoheadrightarrow \Gl$.
The size of this kernel is $\# \Gg / \# \Gl  =[L_{\textup{glob}}:L_{\textup{loc}}]$. Now the result follows by the definitions of $\kl$ and $\kl$ in \eqref{eq:kappas}, together with Corollaries~\ref{cor:kappas} and \ref{cor:knot}. 
\end{proof}

\nocite{rosen}

\bibliographystyle{plain}
\bibliography{HasseNorm}

\end{document}